\documentclass[]{article}

\usepackage[british]{babel}

\usepackage{ifpdf}
\usepackage{calc, xspace}
\usepackage{amsmath, amsthm, amsfonts, euscript}
\usepackage[left=3cm,right=2cm]{geometry}

\usepackage[longnamesfirst,round]{natbib}

\usepackage{ccfonts,eulervm} \usepackage[T1]{fontenc}

\usepackage[monochrome]{color}
\usepackage{graphicx}
  \theoremstyle{plain} \newtheorem{thm}{Theorem}
  \newtheorem{lem}[thm]{Lemma} \theoremstyle{definition}
  \newtheorem{defn}[thm]{Definition} \theoremstyle{remark}

  \newcommand{\Drift}{\operatorname{Drift}}

  \newcommand{\Prob}[1]{\operatorname{\mathbb{P}}\left[#1\right]}
  \newcommand{\Reals}{\mathbb{R}}

  \renewcommand{\d}{\operatorname{d}}
  \DeclareMathOperator{\dist}{dist} 
  \newcommand{\ball}{\operatorname{ball}}
  \newcommand{\diam}{\operatorname{diam}}
  \newcommand{\trace}{\operatorname{tr}}
  \newcommand{\Normal}{\nu}

  \newcommand{\unit}{\underline{\mathbf{e}}}
  \newcommand{\II}{\underline{\underline{\mathbb{I}}}}
  \newcommand{\JJ}{\underline{\underline{\mathbb{J}}}}
  \newcommand{\KK}{\underline{\underline{\mathbb{K}}}}
  \renewcommand{\unit}{\operatorname{\mathbf{e}}}
  \renewcommand{\II}{\operatorname{\mathbb{I}}}
  \renewcommand{\JJ}{\operatorname{\mathbb{J}}}
  \renewcommand{\KK}{\operatorname{\mathbb{K}}}

\begin{document}
%%%%%%%%%%%%%%%%%%%%%%%%%%%%%%%%%%%%%%%%%%%%%%%%%%%%%%%%%%%%%%%%
  %% Title page
  \title{Brownian couplings, convexity, and shy-ness} \author{Wilfrid S. Kendall}
  \date{\text{ }}

 \maketitle

\shortcites{JostKendallMoscoRocknerSturm-1997}
 \begin{abstract}\noindent
\citet*{BenjaminiBurdzyChen-2007} introduced the notion of a \emph{shy coupling}: a coupling of a Markov process such that, for suitable starting points, there is a positive chance of the two component processes of the coupling staying at least a given positive distance away from each other for all time. Among other results, they showed that no shy couplings could exist for reflected Brownian motions in \(C^2\) bounded convex planar domains whose boundaries contain no line segments. Here we use potential-theoretic methods to extend this \citet{BenjaminiBurdzyChen-2007} result (a) to all bounded convex domains (whether planar and smooth or not) whose boundaries contain no line segments, (b) to all bounded convex planar domains regardless of further conditions on the boundary.
 \end{abstract}


 \section{Introduction}\label{sec:introduction}
 \pagenumbering{arabic}
 \setcounter{page}{1}

Motivated by the use of reflection couplings for reflected Brownian motion in a number of contexts (for example efficient coupling as in \citealp{BurdzyKendall-1997a}, and work related to the hotspot conjecture, as in \citealp{AtarBurdzy-2002}, and recent work on Laugesen-Morpurgo conjectures, \citealp{PascuGageonea-2008}), \citet*{BenjaminiBurdzyChen-2007} introduced the notion of shy coupling and studied it in the contexts of Brownian motion on graphs and reflected Brownian motion particularly in convex domains. Shy coupling of two random processes occurs when, for suitable starting points, there is a positive chance of the two component processes of the coupling staying at least a given positive distance away from each other: this of course is in direct contrast to the more usual use of couplings, in which the objective is to arrange for the two processes to meet. Interest in shy couplings is therefore focused on characterizing situations in which shy coupling \emph{cannot} occur. \citet{BenjaminiBurdzyChen-2007} show non-existence of shy coupling for reflected Brownian motions in \(C^2\) bounded convex planar domains whose boundaries contain no line segments.
The present note (influenced by techniques used to study non-confluence of \(\Gamma\protect\)-martingales in \citealp{Kendall-1990a}) uses rather direct potential-theoretic methods to obtain significantly stronger results in the case of reflected Brownian motion in convex domains (Theorems \ref{thm:no-segments} and \ref{thm:planar-case} below). As with \citet{BenjaminiBurdzyChen-2007}, the principal purpose is to contribute to
   a better understanding of probabilistic coupling.

To fix notation, we begin with a formal discussion of coupling and shy coupling.
 In the following \(D\) is a measurable space, and
 the distribution of a Markov process on \(D\) is specified as a
 semigroup of transition probability
measures \(\{P^z_t: t\geq0, z\in D\}\). (The semigroup-based definition allows us to take account of varying starting points, a typical feature of coupling arguments.)
 \begin{defn}[Co-adapted coupling]\label{def:coupling}
   A \emph{coupling} of a Markov process \(\{P^z_t: t\geq0, z\in D\}\)
   on a measurable space \(D\) is a family of random processes \((X,Y)\) on
   \(D^2\), one process
\((X,Y)\) for each pair of starting points \((x_0,y_0)\in D^2\),
   such that \(X\) and \(Y\) each separately are Markov but share the same
  semigroup of transition  probability measures \(\{P^z_t: t\geq0, z\in D\}\). Thus
   for each \(s, t\geq0\) and \(z\in D\) and each measurable
   \(A\subseteq D\) we have
   \begin{equation}
     \begin{split}
       \Prob{X_{s+t}\in A\;|\; X_s=z, \;X_u : 0\leq u\leq
         s}\quad&=\quad
       P^z_t(A)\,,\\
       \Prob{Y_{s+t}\in A\;|\; Y_s=z, \;Y_u : 0\leq u\leq
         s}\quad&=\quad P^z_t(A)\,.
     \end{split}
     \label{eq:coupling}
   \end{equation}
   The coupling is said to be a \emph{co-adapted coupling} if the
   conditioning in \eqref{eq:coupling} can in each case
be extended to include the
   pasts of both \(X\) and \(Y\): for each \(s, t\geq0\) and \(z\in
   D\) and each measurable \(A\subseteq D\)
   \begin{equation}
     \begin{split}
       \Prob{X_{s+t}\in A\;|\; X_s=z, \;X_u, Y_u : 0\leq u\leq
         s}\quad&=\quad
       P^z_t(A)\,,\\
       \Prob{Y_{s+t}\in A\;|\; Y_s=z, \;X_u, Y_u : 0\leq u\leq
         s}\quad&=\quad P^z_t(A)\,.
     \end{split}
     \label{eq:coadapted-coupling}
   \end{equation}
 \end{defn}

 In contrast to \citet{BenjaminiBurdzyChen-2007}'s notion of
 Markovian coupling, we do not require \((X,Y)\) to be Markov. (This
 generalization is convenient but unimportant.) Note that the
 couplings in \citet{BenjaminiBurdzyChen-2007} are all Markovian
 and hence co-adapted.

 We say that \((X, Y)\) begun at \((x_0,y_0)\) \emph{couples
   successfully} on the event 
   \[
   [X_t=Y_t \text{ for all sufficiently
   large }t]\,.
   \]
    Note that a co-adapted coupling can be adjusted on the
 simpler event \([X_t=Y_t \text{ for some }t]\) so as to couple successfully on
 the new event; this need not be the case for more general couplings.
In the remainder of this note the term ``coupling'' will always be short for
``co-adapted coupling''.

 \citet{BenjaminiBurdzyChen-2007}'s notion of
 \emph{shy coupling} for a Markov process on \(D\) is primarily concerned with the
 cases of Brownian motion on graphs and reflected Brownian motion on
 domains in Euclidean space. However their definition is expressed in general terms: suppose that \(D\) is actually a metric
 space equipped with a distance \(\dist\) which furnishes a Borel
 measurability structure:
 \begin{defn}[Shy coupling]
   A coupling \((X, Y)\) is \emph{shy} if there exist two distinct
   starting points \(x_0\neq y_0\) such that for some
   \(\varepsilon>0\)
   \begin{equation}
     \label{eq:shy-coupling}
     \Prob{\dist(X_t, Y_t)>\varepsilon \text{ for all }t\;|\; X_0=x_0, Y_0=y_0}
     \quad>\quad0\,.
   \end{equation}
   In words, \(X\), \(Y\) has positive chance of failing to
   \(\varepsilon\)-couple for some \(\varepsilon>0\) and some pairs of
   starting points.\\
 We say that \((X, Y)\) begun at \((x_0,y_0)\) \emph{\(\varepsilon\)-couples} (for \(\varepsilon>0\))
on the event
\[
[\dist(X_t,Y_t)=\varepsilon \text{ for some }t]\,.
\]
 \end{defn}
 \noindent
 We focus on reflecting Brownian motion in a bounded convex domain in finite-dimensional Euclidean space:
 \begin{defn}[Reflecting Brownian motion]
   A reflecting Brownian motion in the closure \(\overline{D}\) of a
   bounded convex domain \(D\subset \Reals^n\), begun at \(x_0\in D\),
   is a Markov process \(X\) in \(\overline{D}\) solving
   \begin{equation}
     \label{eq:reflecting}
     X_t \quad=\quad x_0 + B_t + \int_0^t \Normal(X_s) \d L^X_s\,,
   \end{equation}
   where \(B\) is \(n\)-dimensional standard Brownian motion, \(L^X\)
   measures local time of \(X\) accumulated at the boundary \(\partial
   D\), and \(\Normal\) is a choice of an inward-pointing unit normal
   vectorfield on \(\partial D\).
 \end{defn}
 \noindent
  Unique solutions of \eqref{eq:reflecting} exist in the case of bounded convex domains \cite[Theorem 3.1]{Tanaka-1979}; this
is a consequence of results about the deterministic Skorokhod's equation.
 Note that reflected Brownian motion can be defined as a semimartingale
 for domains which are much more general than convex domains (see for
 example the treatment of Lipschitz domains in \citealp{BassHsu-1990});
 however the results of this paper are concerned entirely with the convex case.

 Recall that a convex domain in Euclidean spaces can be viewed as the
 intersection of countably many half-spaces.  Consequently the
 inward-pointing unit normal vectorfield \(\Normal\) is unique up to a
 subset of \(\partial D=\overline{D}\setminus D\) of zero Hausdorff
 \((n-1)\)-measure.  The local time term \(L^X\) can be defined using the
 Skorokhod construction: it is the minimal non-decreasing process such
 that the solution of \eqref{eq:reflecting} stays confined within
 \(\overline{D}\). It is determined by the choice of
normal vectorfield \(\Normal\).

 Arguments from stochastic calculus show that all co-adapted couplings
 of such reflected Brownian motions can be represented in terms of
 co-adapted couplings \((A,B)\) of \(n\)-dimensional standard Brownian
 motions, which in turn must satisfy
 \[
 A_t \quad=\quad \int_0^t \JJ_s^\top \d
 B_s + \int_0^t \KK_s^\top \d C_s\,,
 \]
 where \(C\) is an \(n\)-dimensional standard Brownian motion
 independent of \(B\), and \(\JJ\),
 \(\KK\) are \((n\times
 n)\)-matrix-valued random processes, adapted to the filtration
 generated by \(B\) and \(C\), and satisfying the identity
 \[
 \JJ^\top\,\JJ
 +
 \KK^\top\,\KK
 \quad=\quad\II^{(n)} \text{ (the
   identity matrix on \(\Reals^n\))}\,.
 \]
A proof of this fact can be found in passing on page 297 of \citet{Emery-2005}.
An explicit statement and a slightly more direct proof is to be found below at Lemma \ref{lem:well-known}.

 Thus we will study \(X\) and \(Y\) such that
 \begin{equation}
   \label{eq:reflected-coupling}
   \begin{split}
     &X_t \quad=\quad x_0 + B_t + \int_0^t \Normal(X_s) \d L_s^X\,,\\
     &Y_t \quad=\quad y_0 + \int_0^t
     \JJ_s^\top \d B_s + \int_0^t
     \KK_s^\top \d C_s
     + \int_0^t \Normal(Y_s) \d L_s^Y\,,\\
     &\JJ_t^\top\,\JJ_t
     +\KK_t^\top\,\KK_t
     \quad=\quad\II^{(n)}\,.
   \end{split}
 \end{equation}
 The results of this note concern all possible co-adapted couplings,
 of which two important examples are:
 \begin{enumerate}
 \item the \emph{reflection coupling}:\\
   \(\KK=0\) and
   \(\JJ_t=
   \II^{(n)}-2
   \unit_t\unit_t^\top\), where
   \(\unit_t=(X_t-Y_t)/\|X_t-Y_t\|\), so that \(\d A\) is the
   reflection of \(\d B\) in the hyperplane bisecting the segment from
   \(X_t\) to \(Y_t\);
 \item[] and its opposite,
 \item the \emph{perverse coupling}:\\
   \(\KK=0\) and
   \(\JJ_t=-\II^{(n)}+2
   \unit_t\unit_t^\top\), so that the distance
   \(\|X-Y\|\) has trajectories purely of locally bounded variation.
   (This is related to the uncoupled construction in \citealp[Exercise 5.43]{Emery-1989} of the planar process
   \(Z^{\prime\prime}\) with deterministic radial part.)
 \end{enumerate}
 In practice, as is typically the case when studying Brownian
 couplings, general proofs follow easily from the special case when
 \(\KK=0\) and
 \(\JJ\) is an orthogonal matrix. (Heuristic remarks related to this observation are to be found in \citealp[Section
 2]{Kendall-2005}.)

 Note that the existence question for reflection couplings in
convex domains is non-trivial
 (\citealt{AtarBurdzy-2004} establish existence for the more general case of
lip domains), but
 is not relevant for our purposes.

 \citet{BenjaminiBurdzyChen-2007} use ingenious arguments based
 on ideas from differential games to show that a bounded convex planar
 domain cannot support any shy couplings of reflected Brownian motions if
 the boundary is \(C^2\) and contains no line segments
 \citep[Theorem 4.3]{BenjaminiBurdzyChen-2007}. Here we
 describe the use of rather direct potential-theoretic methods (summarized in
 section \ref{sec:potential}); in section
 \ref{sec:results} these are used to generalize the
 \citet{BenjaminiBurdzyChen-2007} result to the cases of (a) all
 bounded convex domains in Euclidean space of whatever dimension with
 boundaries which need not be smooth but must contain no line segments
 (Theorem \ref{thm:no-segments}), and (b) all bounded convex planar
 domains (Theorem \ref{thm:planar-case}). Further extensions and conjectures are discussed in section
 \ref{sec:conclusion}.

% \TBC{Motivation needed at end of this section.  Refer to
%   \protect\citet{BenjaminiBurdzyChen-2007} motivations. Citation for
%   very general definition of reflected Brownian motion (Ruth
%   Williams?). Revise and add gloss.}

 \section{Some lemmas from stochastic calculus and potential theory}\label{sec:potential}
 The following two lemmas from probabilistic potential theory are
 fundamental for our approach.
 \begin{lem}\label{lem:potential-1}
   For \(D\subset \Reals^n\) a bounded domain, for fixed
   \(\varepsilon>0\), consider the closed and bounded (and therefore
   compact) subset of \(\Reals^n\times\Reals^n\) given by
   \[
   F \quad=\quad (\overline{D}\times\overline{D})\setminus \{(x,y):
   \dist(x,y)<\varepsilon\}\,.
   \]
   Suppose there is a continuous function \(\Psi : F\to\Reals\) such
   that the following random process \(Z\) is a
   supermartingale for \emph{any} co-adapted coupling of reflecting
   Brownian motions \(X\) and \(Y\) in the closure \(\overline{D}\),
   when \(S\) is the exit time of \((X, Y)\) from \(F\):
   \[
   Z_t \quad=\quad \Psi(X_{t\wedge S}, Y_{t\wedge S}) + t\wedge S\,.
   \]
   Then for any such coupling almost surely \(S<\infty\) and \(\dist(X_S,
   Y_S)=\varepsilon\).
 \end{lem}
 \begin{proof}
   Being a continuous function on a compact set
   \(F\), \(\Psi\) must be bounded. Thus for any co-adapted coupling \((X,Y)\) the
   supermartingale \(Z\) is bounded below, and so almost
   surely it must converge to a finite value at time \(\infty\). But
   boundedness of \(\Psi\) also means that \(Z\) must almost surely
   tend to infinity on the event \([S=\infty]\). These two
   requirements on \(Z\) force the conclusion that
   \(\Prob{S<\infty}=1\). Since \(X\) and \(Y\) each reflect off \(\partial
   D\), it follows that the exit point \((X_S,Y_S)\) must belong to the part of
   \(\partial F\) which is contained in the boundary of \(\{(x,y):
   \dist(x,y)<\varepsilon\}\); hence almost surely \(\dist(X_S,
   Y_S)=\varepsilon\).
 \end{proof}
 \noindent
 Consequently, if one can exhibit such a \(\Psi\) for a specified \(D\) then any
 coupled \(X\), \(Y\) in \(\overline{D}\) must
 \(\varepsilon\)-couple.  Indeed the existence of such a \(\Psi\)
 implies a uniform bound on exponential moments of the
 \(\varepsilon\)-coupling time \(S\).

 As implied by \citet[Example 4.2]{BenjaminiBurdzyChen-2007}, the
 obvious possibility \(\Psi=c\|X-Y\|^\alpha\) (for positive \(c\),
 \(\alpha\) with \(\alpha\) small) is not suitable here; the perverse
 coupling of the previous section is an example for which all such
 \(\Psi\) lead to \(\Psi(X_{t\wedge S}, Y_{t\wedge S}) + t\wedge S\)
 being a strictly increasing process when neither \(X\) nor \(Y\)
 belong to \(\partial D\). Nevertheless we will see in the next
 section how to construct \(\Psi\) which work simultaneously for all
 co-adapted couplings in a wide range of bounded convex domains.

 Before embarking on this, it is convenient to introduce a further
 lemma motivated by \citet{Ito-1975}'s approach to stochastic
 calculus using modules of stochastic differentials: if \(Z=M+A\) is
 the Doob-Meyer decomposition of a semimartingale \(Z\) into local
 martingale \(M\) and process \(A\) of locally bounded variation, then
 write \(\Drift \d Z=\d A\) and write \((\d Z)^2=\d[M,M]\) for the differential
 of the (increasing) bracket process of \(M\).
 \begin{lem}\label{lem:potential-2}
   The conclusion of Lemma \ref{lem:potential-1} holds if there is a
   continuous \(\Phi:F\to\Reals\) such that, for all co-adapted
   couplings \((X, Y)\) as above,
   \begin{itemize}
   \item[] \(Z=\Phi(X,Y)\) is a semimartingale;
   \item[] moreover the stochastic differential \(\d Z\) satisfies the
     following \emph{random measure} inequalities (with \(a\),
     \(b>0\)):
     \begin{enumerate}
     \item \textbf{(Volatility bounded from below)} \((\d Z)^2 > a \d t\) up to time \(S\);\label{req:1}
     \item \textbf{(Drift bounded from above)} \(\Drift \d Z < b \d t\) up to time \(S\).\label{req:2}
     \end{enumerate}
   \end{itemize}
 \end{lem}
 \begin{proof}
   For \(c, \lambda>0\) to be chosen at the end of the proof, set \(\Psi(x,y)=c (1-
   \exp(-\lambda \Phi(x,y)))\) and apply It\^o's lemma to
   \(\Psi(X,Y)=c(1-\exp(-\lambda Z))\) up to the exit time \(S\):
   \begin{multline*}
     \Drift \d \Psi(X,Y)\quad=\quad c\lambda\exp(-\lambda
     Z)\left(\Drift
       \d Z - \tfrac{1}{2} \lambda (\d Z)^2\right)
%\\
     \quad\leq\quad - \tfrac{c}{2}\lambda\exp(-\lambda Z)\left(\lambda
       a - 2b\right)\d t
   \end{multline*}
   (where the inequality is viewed as an inequality for random
   measures). Fix \(\lambda > 2b/a\) and \(c>2\lambda^{-1}\tfrac{\exp(\lambda
   \max\{\Phi\})}{\lambda a-2b}\). The above calculation then shows
   that \(\Psi(X_{t\wedge S},Y_{t\wedge S})+(t\wedge S)\) is a
   supermartingale, so the conclusion of Lemma \ref{lem:potential-1}
   holds.
 \end{proof}

 %\TBC{Revise these lemmas from potential theory.}

As mentioned above in section \ref{sec:introduction}, 
the following stochastic calculus result is to be found unannounced and in passing on page 297 of \citet{Emery-2005}.
It provides an explicit representation for co-adaptively coupled Brownian motions.
We state it here as a lemma and indicate a direct proof in order to establish an explicit statement of the result in
the literature.
\begin{lem}\label{lem:well-known}
Suppose that \(A\) and \(B\) are two co-adapted \(n\)-dimensional Brownian motions. Augmenting the filtration if necessary by adding an independent adapted \(n\)-dimensional Brownian motion \(D\), it is possible to construct an adapted \(n\)-dimensional Brownian motion \(C\), independent of \(B\), such that
\[
 A \quad=\quad \int \JJ^\top \d
 B + \int \KK^\top \d C\,,
 \]
 where \(\JJ\),
 \(\KK\) are \((n\times
 n)\)-matrix-valued predictable random processes, satisfying the identity
 \[
 \JJ^\top\,\JJ
 +
 \KK^\top\,\KK
 \quad=\quad\II^{(n)} \text{ (the
   identity matrix on \(\Reals^n\))}\,.
 \]
\end{lem}
\newcommand{\HH}{\mathbb{H}}
\begin{proof}
Certainly the quadratic covariation between the vector semimartingales \(A\) and \(B\) may be expressed as \(\d A\,\d B^\top=\JJ^\top\d t\) for a predictable \((n\times
 n)\)-matrix-valued predictable random process \(\JJ\) such that the symmetric matrix inequality \(0\leq \JJ^\top\JJ\leq \II^{(n)}\) holds in the spectral sense (\(0\leq x^\top\JJ^\top\JJ x\leq x^\top x\) for all vectors \(x\)); this is a consequence of the Kunita-Watanabe inequality. Since \(\JJ^\top\JJ\) is a symmetric contraction matrix we find \((\JJ^\top\JJ)^k\to\HH_1\) as \(k\to\infty\), where \(\HH_1\) is the orthogonal projection onto the null-space of \(\JJ^\top\JJ-\II^{(n)}\). This exhibits \(\HH_1\) as a measurable function of \(\JJ\). We may now extract \(\lambda_2^2=\sup\{x^\top(\JJ^\top\JJ-\HH_1)x:\|x\|=1\}\) as a measurable function of \(\JJ\), renormalize and study \((\JJ^\top\JJ-\HH_1)^k/\lambda_2^{2k}\to\HH_2\), and continue so as to represent
\begin{equation}\label{eq:spectral-decomposition}
\JJ^\top\JJ \quad=\quad \HH_1 + \sum_{i=2}^N \lambda_i^2 \,\HH_i\,,
\end{equation}
where \(\HH_1\), \(\HH_2\), \ldots, \(\HH_N\) are disjoint orthogonal projections, \(N\leq n\), and \(1>\lambda_2>\ldots>\lambda_N>0\) and all quantities are measurable functions of \(\JJ\).

Set \(\HH_0=\II^{(n)}-\HH_1 - \HH_2 - \ldots - \HH_N\) and define
\[
\KK \quad=\quad \HH_0 + \sum_{i=2}^N \sqrt{1-\lambda_i^2}\,\HH_i
\]
as the non-negative symmetric square root of \(\II^{(n)}-\JJ^\top\JJ\). The above spectral approach defines \(\KK\) as a measurable function of the matrix \(\JJ\), and hence we may view \(\KK\) as well as \(\JJ\) as predictable \((n\times
 n)\)-matrix-valued predictable random processes, now satisfying \(\JJ^\top\,\JJ
 +
 \KK^\top\,\KK
=\II^{(n)}\).

Moreover we may construct a pseudo-inverse of \(\KK\) as a further predictable symmetric process:
\[
\KK^+\quad=\quad \HH_0 + \sum_{i=2}^N \frac{1}{\sqrt{1-\lambda_i^2}}\,\HH_i\,,
\]
so that \(\KK^+\,\KK=\KK\,\KK^+=\II^{(n)}-\HH_1\).

Now define \(C\) by
\[
\d C\quad=\quad \KK^+(\d A - \JJ^\top\d B) + \HH_1\d D\,.
\]
The quadratic variation of \(\d A - \JJ^\top\d B\) is
\[
(\d A - \JJ^\top\d B)(\d A - \JJ^\top\d B)^\top\quad=\quad (\II^{(n)}-\JJ^\top\JJ)\d t=\KK^\top\KK\d t\,,
\]
and so the stochastic differential \(\KK^+(\d A - \JJ^\top\d B)\) has finite quadratic variation \((\II^{(n)}-\HH_1)\d t\)
and therefore (by the \(L^2\) theory of stochastic differentials of continuous semimartingales) it is a martingale differential.
Hence \(C\) is a continuous martingale with quadratic variation
\(\d C\,\d C^\top=\II^{(n)}\d t\), by which we may deduce that \(C\) is \(n\)-dimensional Brownian motion.
Moreover \(\d C\,\d B^\top=\KK^+(\d A - \JJ^\top\d B)\d B^\top=\KK^+(\JJ^\top-\JJ^\top)\d t=0\), so \(C\) is independent of \(B\).

Finally \(\KK^\top\d C = (\II^{(n)}-\HH_1)(\d A-\JJ^\top\d B)\) (use \(\KK^\top\HH_1=0\)) and \(\HH_1(\d A-\JJ^\top\d B)\) is a martingale differential with quadratic variation
\begin{multline*}
 \HH_1(\d A-\JJ^\top\d B)(\d A-\JJ^\top\d B)^\top\HH_1
\quad=\quad
\HH_1(\II^{(n)}-\JJ^\top\JJ)\HH_1\d t
%\\
\quad=\quad 
\HH_1\KK^\top\KK\HH_1\d t\quad=\quad0\,.
\end{multline*}
So \(\KK^\top\d C=\d A-\JJ^\top\d B\), which establishes the result.
\end{proof}
\citet{Emery-2005} uses an approximation argument, which we circumvent here by exhibiting \(\KK\) as a predictable process via the spectral decomposition \eqref{eq:spectral-decomposition}. Note also that the extra Brownian motion \(D\) is not required if \(\JJ^\top\,\JJ<\II^{(n)}\),
in which case \(\HH_1=0\).

 \section{Theorems and proofs}\label{sec:results}
 The first theorem generalizes the planar result of \citet[Theorem
 4.3]{BenjaminiBurdzyChen-2007} to the case of higher dimensions, and
 removes the requirement of boundary smoothness.
 \begin{thm}\label{thm:no-segments}
   Let \(\overline{D}\) be the closure of a bounded convex domain in \(\Reals^n\) such that
   \(\partial D\) contains no line segments. Then no co-adapted
   coupling of reflected Brownian motions in \(\overline{D}\) can be shy.
 \end{thm}
 \begin{proof}
   It suffices to exhibit, for any fixed \(\varepsilon\), a
   function \(\Phi\) satisfying the two requirements of Lemma
   \ref{lem:potential-2}. Motivated by a similar construction used to
establish
the convex geometry of small
   hemispheres \citep{Kendall-1991c}, define (for \(p\not\in D\) and
   \(\delta>0\))
   \begin{multline}
     \label{eq:twisted-distance}
     V_{p,\delta}(x,y)
\;=\;
     \tfrac{1}{2} \|x-y\|^2 + \tfrac{\delta}{2}\|x-p\|^2-\tfrac{\delta}{2}\|y-p\|^2
     \;=\;
 \tfrac{1}{2} \|x-y\|^2 + \delta \langle
     x-y,\tfrac{x+y}{2}-p\rangle\,.
   \end{multline}
   Thus \(V_{p,\delta}\) is a hyperbolic perturbation of
   \(\tfrac{1}{2}\|x-y\|^2\) based on the \emph{pole} \(p\).  For any
   fixed \(\varepsilon>0\) and for all sufficiently small \(\delta>0\)
   depending on \(\varepsilon\), \(p\), and the geometry of \(D\) we show that
   \(\Phi=V_{p,\delta}\) satisfies the requirements of Lemma
   \ref{lem:potential-2}. The result then follows by applying Lemmas
   \ref{lem:potential-2} then \ref{lem:potential-1}.

The first step is to establish Lemma \ref{lem:potential-2}
requirement \ref{req:1} \textbf{(volatility bounded from below)}.
   Fixing \(\varepsilon>0\), suppose by virtue of Lemma \ref{lem:well-known}
that \(X\), \(Y\) satisfy
   \eqref{eq:reflected-coupling} for some co-adapted matrix processes
   \(\JJ\),
   \(\KK\). Applying It\^o's lemma and
   the fact that \(B\) and \(C\) are independent \(n\)-dimensional
   standard Brownian motions,
   \begin{multline}
     (\d \Phi(X,Y))^2\quad=\quad
%\\
     \left(\|X-Y+\delta(X-p) -
       \JJ(X-Y+\delta(Y-p))\|^2 +
       \|\KK(X-Y+\delta(Y-p))\|^2\right) \d t\\
     \quad\geq\quad \|X-Y+\delta(X-p) -
     \JJ(X-Y+\delta(Y-p))\|^2
     \d t\\
     \quad\geq\quad \left(\|X-Y+\delta(X-p)\| -
       \|X-Y+\delta(Y-p)\|\right)^2 \d t\,,
   \end{multline}
   since the third equation of \eqref{eq:reflected-coupling} implies
   that the linear map \(\JJ\) is a contraction.
   Suppose now that \(\|X-Y\|>\varepsilon\). If (for example)
   \begin{equation}\label{eq:first-delta-bound}
     \delta \quad<\quad \frac{\varepsilon}{2\sup\{\dist(p,w):w\in D\}}
   \end{equation}
   then
   \begin{multline}
     \|X-Y+\delta(X-p)\|^2\quad=\quad \|X-Y+\delta(Y-p)+\delta(X-Y)\|^2\\
     \quad=\quad \|X-Y+\delta(Y-p)\|^2
     +2\delta\|X-Y\|^2 + 2\delta^2\langle X-Y, \tfrac{X+Y}{2}-p\rangle\\
\quad\geq\quad
\|X-Y+\delta(Y-p)\|^2+2\delta\|X-Y\|^2 
- 2\delta^2|\langle X-Y, \tfrac{X+Y}{2}-p\rangle|
\\
\quad=\quad
\|X-Y+\delta(Y-p)\|^2
+2\delta\|X-Y\|^2%
\left(1-\delta\frac{\left|\langle \tfrac{X-Y}{\|X-Y\|},%
\tfrac{X+Y}{2}-p\rangle\right|}%
{\|X-Y\|}\right)
\\
     \quad\geq\quad \|X-Y+\delta(Y-p)\|^2 + \delta\varepsilon^2\,.
   \end{multline}
   This in turn implies
   \begin{multline}
     \|X-Y+\delta(X-p)\| - \|X-Y+\delta(X-Y)\| \quad\geq\quad
\\
\quad\geq\quad
     \frac{\delta\varepsilon^2}{\|X-Y+\delta(X-p)\| + \|X-Y+\delta(X-Y)\|}\\
     \quad\geq\quad \frac{\delta\varepsilon^2}%
     {(2+\delta)\diam{D}+\delta\sup\{\dist(p,w):w\in D\}}\,.
   \end{multline}
   Thus for all small enough \(\delta>0\) (bounded by \eqref{eq:first-delta-bound}) there
   is a constant \(a=a(\varepsilon,\delta,p,D)>0\) such that \(\d
   \Phi(X,Y))^2>a\d t\) while \(\|X-Y\|>\varepsilon\), no matter what
   co-adapted coupling is employed.  This establishes requirement
   \ref{req:1} of Lemma \ref{lem:potential-2}. Note that we have not
   yet used the condition that \(\partial D\) be free of line
   segments.

The second step is to establish Lemma \ref{lem:potential-2}
requirement \ref{req:2} \textbf{(drift bounded from above)}.
   From \eqref{eq:reflected-coupling},
   \begin{multline}
     \Drift \d \Phi(X,Y)\quad=\quad \\
\langle X-Y + \delta(X-p),
     \Normal(X)\rangle\d L^X +
     \langle Y-X - \delta(Y-p), \Normal(Y)\rangle\d L^Y
     + (n -
     \tfrac{1}{2}\trace(\JJ+\JJ^\top)\d
     t\,.
   \end{multline}
   Since
   \(\trace(\JJ+\JJ^\top)\geq
   -2n\), requirement \ref{req:2} of Lemma \ref{lem:potential-2} is
   established with \(b=2n\) if we can show, for \(x\), \(y\) in \(\overline{D}\)
   with \(\|x-y\|\geq\varepsilon\),
   \begin{equation}
     \begin{split}
       \langle x-y + \delta(x-p), \Normal(x)\rangle \leq 0 & \qquad
       \text{ when } x\in\partial D\,,
       \\
       \langle y-x - \delta(y-p), \Normal(y)\rangle \leq 0 & \qquad
       \text{ when } y\in\partial D\,.
     \end{split}
\label{eq:strict-convexity}
   \end{equation}
   \noindent
Both inequalities hold for all small enough \(\delta>0\) depending on \(\varepsilon\), \(p\), and the geometry of \(D\). Consider for example
the second inequality in \eqref{eq:strict-convexity}. The Euclidean set
   \[
   K\;=\;
   \{(x,y,v): x\in \overline{D}, y\in \partial D, v \text{ is
     unit inward-pointing normal to } \partial D \text{ at } y\}\,,
   \]
   is closed and bounded hence compact.  By convexity of \(D\) the inner
   product \(\langle y-x, v\rangle\) is non-positive on the subset \(K\) and can vanish
on \(K\) only when \(x=y\) or the segment
   between \(x\) and \(y\) lies in \(\partial D\). Moreover this inner
   product is a continuous function on \(K\).  But
   \(\partial D\) contains no line segments and so \(\langle y-x,
   v\rangle\) is negative everywhere on the compact set
   \(K\setminus\{(x,y,v):\|x-y\|<\varepsilon\}\), and therefore must
   satisfy a negative upper bound. Thus the second inequality of
   \eqref{eq:strict-convexity} also holds
   when \(\|x-y\|\geq\varepsilon\), for all small enough \(\delta\)
   depending on \(\varepsilon\), \(p\), and the geometry of \(D\).
The first inequality follows similarly.

   Hence both requirements of Lemma \ref{lem:potential-2} apply for
   \(\Phi=V_{p,\delta}\) for any fixed \(\varepsilon>0\)
   once \(\delta\) is small enough (depending on \(\varepsilon\) and
   the geometry of \(D\)). It follows from Lemma \ref{lem:potential-1}
   that any co-adapted coupling \(X,Y\) eventually attains
   \(\dist(X,Y)\leq\varepsilon\) for any \(\varepsilon>0\), and hence
   no co-adapted coupling can be shy.
 \end{proof}

\begin{thm}\label{thm:planar-case}
  Let \(D\) be a bounded convex planar domain. Then no co-adapted
  coupling of reflected Brownian motions in \(D\) can be shy.
\end{thm}

\begin{proof}
  The proof strategy is the same as for Theorem \ref{thm:no-segments}:
  exhibit, for any fixed \(\varepsilon>0\), a function \(\Phi\)
  satisfying the requirements of Lemma \ref{lem:potential-2}. However
  the function will now be the minimum
  \(\Phi=\Phi_1\wedge\ldots\wedge\Phi_k\) of a finite number of
  functions \(\Phi_i\) of the form \(V_{p,\delta}\) or a mild
  generalization thereof. The Tanaka formula \citep[VI Section
  1.2]{RevuzYor-1991} shows that if each \(\Phi_i\) satisfies
  requirement \ref{req:1} of Lemma \ref{lem:potential-2}, and
  satisfies requirement \ref{req:2} on \(\Phi_i=\Phi\), then \(\Phi\)
  satisfies both requirements and so we may apply Lemmas
  \ref{lem:potential-2} and then \ref{lem:potential-1} as above.

  The mild generalization modifies \(V_{p,\delta}(x,y)\) by a small
  perturbation. This is required in order to deal with the possibility that \(\partial D\) contains  parallel line segments.
\begin{multline}
  \label{eq:perturbedV}
  \widetilde{V}_{p,\delta,S}(x,y)\quad=\quad
\tfrac{1}{2}\|x-y\|^2
+\tfrac{\kappa}{2}\left(\|x-p\|^2-\|y-p\|^2\right)
\\
\quad=\quad
\tfrac{1}{2}\|x-y\|^2
+\kappa\langle x-y,\tfrac{x+y}{2}-p\rangle\,,\\
\text{ where now } \kappa =
\delta\exp\left(-\tfrac{1}{S}\|\tfrac{x+y}{2}-p\|\right)\,.
\end{multline}
Thus the asymmetry of the hyperbolic perturbation
\(\widetilde{V}_{p,\delta,S}\) now depends on the distance of the
mid-point \(\tfrac{x+y}{2}\) from the pole \(p\).

The first step is to establish Lemma \ref{lem:potential-2} requirement
\ref{req:1} \textbf{(volatility bounded from below)}.  The argument
runs much as for Theorem \ref{thm:no-segments}, but is further complicated by the
need to deal with parallel line segments in \(\partial D\).  For
convenience we introduce the notation
\begin{equation}\label{eq:notation}
  Z \quad=\quad \tfrac{X+Y}{2}-p\,,\qquad\qquad
  \unit \quad=\quad Z/\|Z\|\,.
\end{equation}
Thus (using the fact that \(\JJ\) is a
contraction)
\begin{multline*}
(\d \widetilde{V}_{p,\delta,S}(X,Y))^2 / \d t\quad=\quad\\
\|
X-Y+\kappa(X-p-\tfrac{1}{2S}\langle X-Y,Z\rangle\unit)
-
\JJ(
X-Y+\kappa(Y-p+\tfrac{1}{2S}\langle X-Y,Z\rangle\unit)
)
\|^2+\\
+
\|
\KK(
X-Y+\kappa(Y-p+\tfrac{1}{2S}\langle X-Y,Z\rangle\unit)
)
\|^2
\\
%\quad\geq\quad
%\|
%X-Y+\kappa(X-p-\tfrac{1}{2S}\langle X-Y,Z\rangle\unit)
%-
%\JJ(
%X-Y+\kappa(Y-p+\tfrac{1}{2S}\langle X-Y,Z\rangle\unit)
%)
%\|^2\\
\quad\geq\quad
\Big(
\|X-Y+\kappa(X-p-\tfrac{1}{2S}\langle X-Y,Z\rangle\unit)\|
-
\|X-Y+\kappa(Y-p+\tfrac{1}{2S}\langle X-Y,Z\rangle\unit)\|
\Big)^2
\,.
\end{multline*}
\noindent
Now consider the difference of the squared norms of the summands:
if
\begin{align}\label{eq:conditions-1}  
S &\quad>\quad \sup\{\dist(p,w):w\in D\}\,,\\
\kappa\;\leq\;\delta&\quad<\quad\varepsilon/\sup\{\dist(p,w):w\in D\}
\label{eq:conditions-2}  
\end{align}
(so that in particular \(1-\tfrac{\|Z\|}{S}>0\)) then
\begin{multline*}
\|X-Y+\kappa(X-p-\tfrac{1}{2S}\langle X-Y,Z\rangle\unit)\|^2
-
\|X-Y+\kappa(Y-p+\tfrac{1}{2S}\langle X-Y,Z\rangle\unit)\|^2\\
\;=\;
2\kappa\langle
X-Y+\kappa Z,
X-Y-\tfrac{1}{S}\langle X-Y,Z\rangle\unit
\rangle
\\
\;=\;
2\kappa\left(\|X-Y\|^2
-
\tfrac{\|Z\|}{S}\langle X-Y,\unit\rangle^2
+\kappa \|Z\|(1-\tfrac{\|Z\|}{S})\langle X-Y,\unit\rangle\right)\\
\;\geq\;
2\kappa\|X-Y\|^2\left(1-\tfrac{\|Z\|}{S}\right)
-2\kappa\left(\kappa \|Z\|(1-\tfrac{\|Z\|}{S})|\langle X-Y,\unit\rangle|\right)
\\
\;\geq\;
2\kappa\|X-Y\|^2\left(1-\tfrac{\|Z\|}{S}\right)
\left(
1-\kappa\|Z\|\tfrac{|\langle X-Y,\unit\rangle|}{\|X-Y\|^2}
\right)
\\
\quad\geq\quad
2\kappa\|X-Y\|^2\left(1-\tfrac{\sup\{\dist(p,w):w\in D\}}{S}\right)\left(1
-\kappa\frac{\sup\{\dist(p,w):w\in D\}}{\varepsilon}\right)
\,.
\end{multline*}
It follows from inequalities \eqref{eq:conditions-1} and \eqref{eq:conditions-2} that we obtain a positive lower bound on
\[
\|X-Y+\kappa(X-p-\tfrac{1}{2S}\langle X-Y,Z\rangle\unit)\|
-
\|X-Y+\kappa(Y-p+\tfrac{1}{2S}\langle X-Y,Z\rangle\unit)\|
\]
subject to the further condition (required to obtain a lower bound on each of
the two norms in the difference above)
\begin{equation}
  \label{eq:conditions-3}
\kappa\;\leq\;\delta\quad<\quad
\varepsilon/\left(\sup\{\dist(p,w):w\in D\}+\tfrac{1}{2}\diam(D)\right)\,.
\end{equation}

Now, for a finite set of poles \(p^\pm_i\), to be determined below in the ``localization argument'',  consider
\begin{equation}\label{eq:minimum-1}
\Phi(x,y)\;=\;
 \bigwedge_{i=1}^{m} \left(\widetilde{V}_{p^{+}_i,\delta,S}(x,y)\wedge \widetilde{V}_{p^{+}_i,\delta,S}(y,x)
   \wedge  \widetilde{V}_{p^{-}_i,\delta,S}(x,y)\wedge \widetilde{V}_{p^{-}_i,\delta,S}(y,x)\right)\,.
\end{equation}
Subject to \(S\) and \(\kappa\) satisfying the conditions
(\ref{eq:conditions-1}, \ref{eq:conditions-2}, \ref{eq:conditions-3})
as \(p\) runs through the poles \(p^\pm_j\)
% (it suffices to replace
% \(\sup\{\dist(p,w):w\in D\}\) by \(R+\diam(D)\) for large \(R\)),
it then follows from the Tanaka formula that \(\Phi\) satisfies
requirement \ref{req:1} of Lemma \ref{lem:potential-2}.

\begin{figure}[htb]
\includegraphics[width=3in]{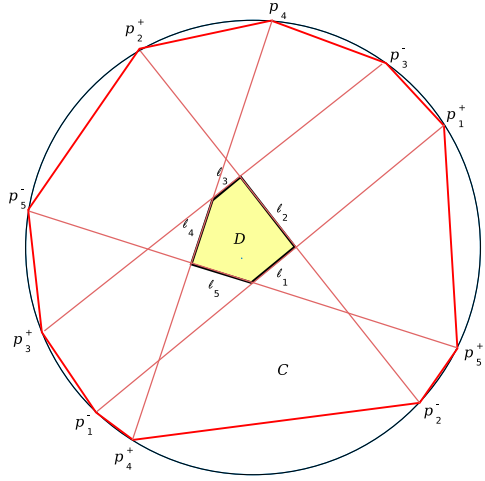}
  \centering
  \caption{Circle \(C\) centred in \(D\) of large radius \(R\), with
    intersection points produced by extending the maximal line
    segments in \(\partial D\) (for simplicity \(D\) is chosen to be a
    polygon). Note that \(\ell_1\) and \(\ell_3\) are parallel.}
\label{fig:figureC}
\end{figure}

The second step is a \textbf{``localization'' argument} aimed
eventually at showing that requirement \ref{req:2} of Lemma
\ref{lem:potential-2} (drift bounded from above) is satisfied on each
locus \(\Phi_i=\Phi\). To this end we must first specify the various poles
\(p^\pm_i\) involved in the different \(\Phi_i\) making up
\(\Phi=\Phi_1\wedge\ldots\wedge\Phi_k\) and identify the region where
\(\Phi_i=\Phi\) (thus, ``localizing'').  For fixed \(\varepsilon>0\)
there can only be finitely many maximal linear segments \(\ell_1\),
\ldots, \(\ell_m\subset\partial D\) of length at least
\(\varepsilon\).  Fix a circle \(C\) centred in \(D\) of radius \(R\),
and locate the poles \(p^{+}_i\), \(p^{-}_i\) at the intersection
points on \(C\) of the line defined by \(\ell_i\), sign \(\pm\) chosen
according to orientation (Figure \ref{fig:figureC}). Here \(R\) must
be chosen to be large enough to fulfil asymptotics given below at
\eqref{eq:asymptotics}: in addition we require that \(R>\diam(D)
\csc(\phi)\), where \(3\phi\) is the minimum modulo \(\pi\) of the
non-zero angles between lines \(\ell_i\), \(\ell_j\). The rationale
for this is that if no other line segments are parallel to a given
\(\ell_i\), and if \(x\), \(y\in\ell_i\), then the term corresponding
to \(\ell_i\) (\emph{via} the poles \(p^\pm_i\)) in the minimum \eqref{eq:minimum-1} is then the unique
minimizer.  For example \(\widetilde{V}_{p^{+}_i,\delta,S}(x,y)\) is
the unique minimizer if \(\kappa \langle
x-y,\tfrac{x+y}{2}-p_i^+\rangle\) is the unique minimum of the
corresponding inner products. Hence \(R>\diam(D)\csc(\phi)\) suffices to
localize in the absence of parallelism, by a simple geometric argument
indicated in Figure \ref{fig:figureG}. This argument uses the remark
(established by calculus) that
\[
\exp\left(-
\frac{1}{S}\|\tfrac{x+y}{2}-p_i^+\|
\right)\langle x-y,\frac{x+y}{2}-p_i^+\rangle
\]
is the minimum if no other \(p_j^\pm\) is separated from \(D\) by the perpendicular to \(\ell_i\) at \(p_i^+\)
(which follows from the constraint of \eqref{eq:conditions-1},
which we have required for all poles \(p\)).

\begin{figure}[htb]
\includegraphics[width=3in]{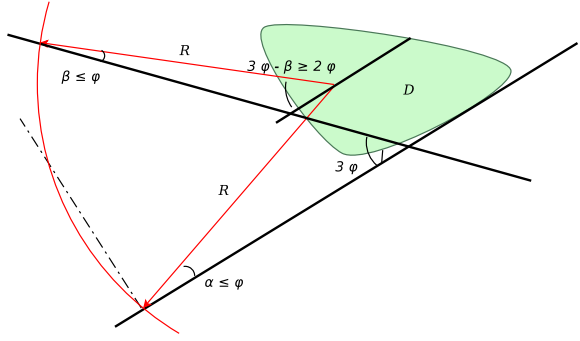}
  \centering
  \caption{ Let \(3\phi\) be the minimum modulo \(\pi\) of the
    non-zero angles between line segments \(\ell_i\), \(\ell_j\).  If
    \(R>\diam(D) \csc(\phi)\) and \(\ell_i\) is parallel to no other
    line segments then the poles \(p^{+}_i\), \(p^{-}_i\) for a given
    \(\ell_i\) can be seen to supply the minimum in \eqref{eq:minimum-1} when
    \(x\), \(y\in\ell_i\).}
\label{fig:figureG}
\end{figure}

In case some of the \(\ell_i\) are parallel, then we must use the
particular features of \(\widetilde{V}_{p,\delta,S}\) as opposed to
\(V_{p,\delta}\), taking into account the expression \eqref{eq:perturbedV}
for \(\kappa\). Suppose that \(x\),
\(y\in\ell_i\) and \(\ell_i\) is parallel to \(\ell_j\). Establish
cartesian coordinates \((u,v)\) centred at \(\tfrac{x+y}{2}\), for
which \(\ell_i\) lies on the \(u\)-axis. Without loss of generality
suppose that \(\ell_j\) lies on the locus \(v=h\). Suppose the circle
\(C\) is centred at \((u_0, v_0)\) lying between \(\ell_i\) and
\(\ell_j\) (Figure \ref{fig:parallel}).

We need to show that, for all large enough \(R\) and all other poles \(p^\pm_j\) for which \(\langle x-y,\frac{x+y}{2}-p_j^\pm\rangle\) has the same sign as
\(\langle x-y,\frac{x+y}{2}-p_i^+\rangle\),
\begin{equation} \label{eq:criterion-2}
  \frac{\exp\left(-
\frac{1}{S}\|\tfrac{x+y}{2}-p_i^+\|
\right)\langle x-y,\frac{x+y}{2}-p_i^+\rangle}%
{\exp\left(
-\frac{1}{S}\|\tfrac{x+y}{2}-p_j^\pm\|
\right)\langle x-y,\frac{x+y}{2}-p_j^\pm\rangle}%
\quad>\quad 1\,,
\end{equation}
and we now show that this will be the case if \(1<\sigma=S/R<2\) for
large enough \(R\) (note that once \(R\) is large enough this is compatible with
\eqref{eq:conditions-1}, which is our other
requirement on \(R\)).

\begin{figure}[htb]
\includegraphics[width=3in]{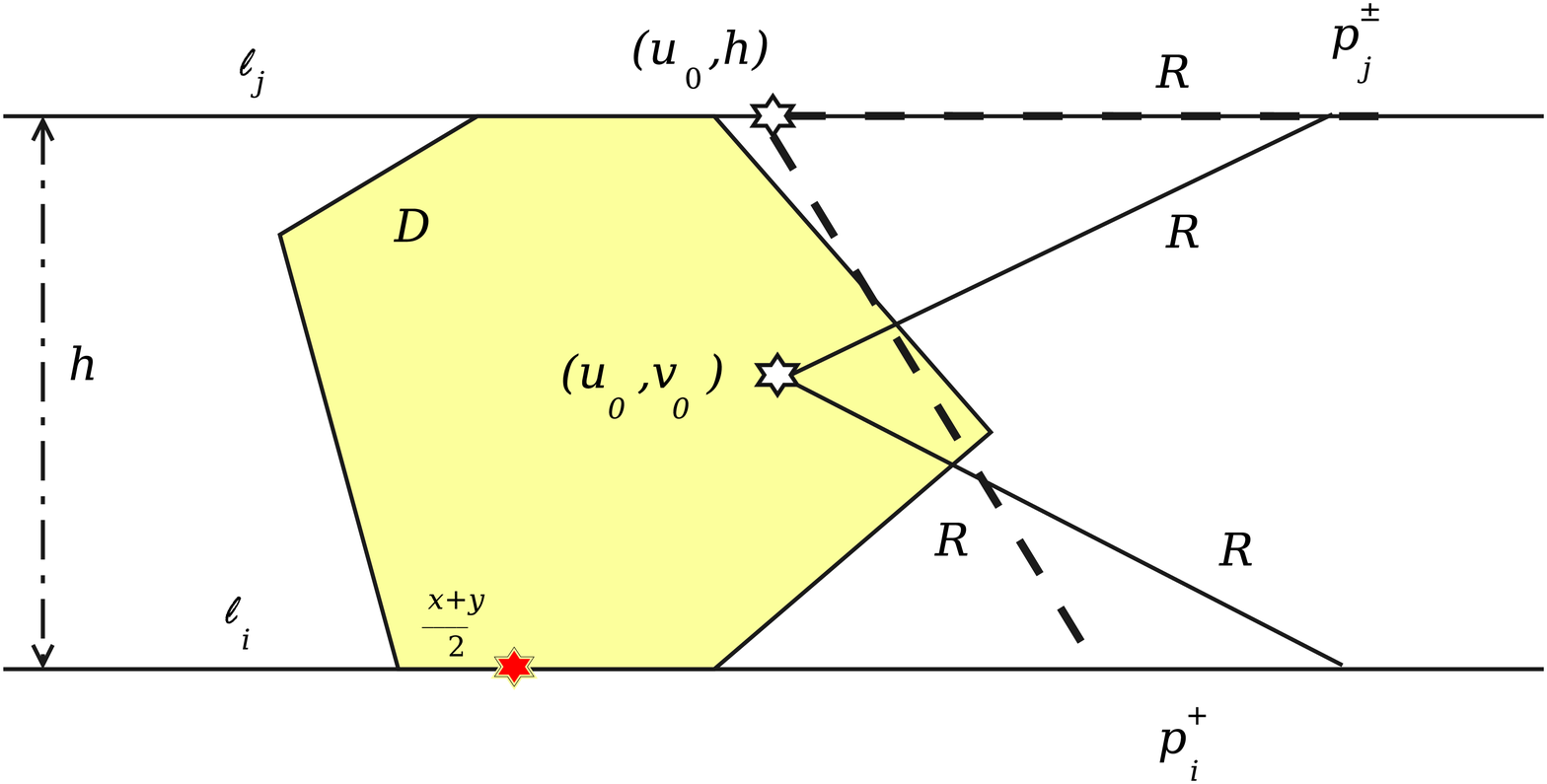}
  \centering
  \caption{ Illustration of the movement of poles \(p_i^+\), \(p_j^\pm\) when the
    centre of \(C\) is moved perpendicularly to the parallel line
    segments \(\ell_i\), \(\ell_j\).  }
\label{fig:parallel}
\end{figure}

First note that it suffice to consider the case when \(v_0\) is as
large as possible. For if the centre of \(C\) is moved say from
\((u_0,v_0)\) to \((u_0,h)\) then the pole \(p_j^\pm\) is moved further out on
\(\ell_j\), while the pole \(p_i^+\) is brought closer in on \(\ell_i\) (see
Figure \ref{fig:parallel}). Calculus shows that
\[
u \exp\left(-\tfrac{1}{S}\sqrt{u^2+v^2}\right)
\]
is increasing in \(u\) for
\(0<u<S\sqrt{\tfrac{1+\sqrt{1+4v^2/S^2}}{2}}\) (true for all feasible
moves when \(\sigma=S/R>1\), since \(R>\diam(D)\csc(\phi)\) and \(0<\phi<\pi/3\)), so it follows that such a move will
increase the denominator and decrease the numerator of
\eqref{eq:criterion-2}.

We now establish asymptotics which are uniform in \((u_0, v_0)\in D\); indeed, uniformly for \(-\diam(D)\leq u_0\leq\diam(D)\) and holding
\(\sigma=S/R\) fixed,
\begin{multline}\label{eq:asymptotics}
    \frac{\exp\left(-
\frac{1}{S}\|\tfrac{x+y}{2}-p_i^+\|
\right)\langle x-y,\frac{x+y}{2}-p_i^+\rangle}%
{\exp\left(
-\frac{1}{S}\|\tfrac{x+y}{2}-p_j^\pm\|
\right)\langle x-y,\frac{x+y}{2}-p_j^\pm\rangle}
\quad\geq\quad
\frac{\sqrt{R^2-h^2}+u_0}{R+u_0}\times
\frac{\exp\left(-\frac{1}{S}(\sqrt{R^2-h^2}+u_0)\right)}%
{\exp\left(-\frac{1}{S}\sqrt{(R+u_0)^2+h^2}\right)}
\\
\;=\;
\left(1-\frac{R-\sqrt{R^2-h^2}}{R+u_0}\right)
\times
\exp\left(\frac{1}{R \sigma}\left(
(R-\sqrt{R^2-h^2})
+
(\sqrt{(R+u_0)^2+h^2}-(R+u_0))
\right)\right)\\
\;\geq\;
\left(
1-\tfrac{R}{R+u_0}(\tfrac{h^2}{2R^2}+o(R^{-2}))
\right)
\times
\exp\left(\frac{1}{\sigma}\left(
\tfrac{h^2}{2R^2}+o(R^{-2})
+
\tfrac{R+u_0}{R}(\tfrac{h^2}{2(R+u_o)^2}+o((R+u_0)^{-2}))
\right)\right) \\
\quad=\quad
1+\left(\tfrac{1}{\sigma}-\tfrac{1}{2}\right)\tfrac{h^2}{R^2}
 + o(\tfrac{1}{R^2})\,.
\end{multline}
It follows from these uniform asymptotics
that if \(1<\sigma<2\) then \eqref{eq:criterion-2} is
satisfied for all large enough \(R\). 

As a consequence of these considerations, and bearing in mind the
continuity properties of the criterion ratio on the left-hand side of
\eqref{eq:criterion-2}, if we fix \(S/R\in(1,2)\) and \(R\) large enough then (noting \(\|x-y\|>\varepsilon\)) it follows that for all sufficiently small
\(\eta>0\),
if \(x\), \(y\in D\cap(\ell_i\oplus\ball(0,\eta))\) and one of \(x\),
\(y\in\partial D\) then the active component of \eqref{eq:minimum-1}
is the one involving \(\ell_i\) \emph{via} the poles \(p_i^\pm\).

The final step is to establish Lemma \ref{lem:potential-2}
requirement \ref{req:2} \textbf{(drift bounded from above)}. When \(p^\pm\) is the active pole, the drift is given by
   \begin{multline*}
\Big\langle X-Y + \kappa(X-p^\pm)
 -\tfrac{\kappa}{2S}\langle X-Y,Z\rangle \unit,
     \Normal(X)\Big\rangle\d L^X 
%+
%\\
+
\Big\langle Y-X - \kappa(Y-p^\pm)
 -\tfrac{\kappa}{2S}\langle X-Y,Z\rangle \unit,
 \Normal(Y)\Big\rangle\d L^Y\\
     + 
(n - \tfrac{1}{2}\trace(\JJ+\JJ^\top))\d t
   \end{multline*}
   (using the notation of the proof of Theorem
   \ref{thm:planar-case}). In order to establish requirement
   \ref{req:2} of Lemma \ref{lem:potential-2}, it suffices to show
   non-positivity of
\[
\Big\langle X-Y + \kappa(X-p^\pm)
 -\tfrac{\kappa}{2S}\langle X-Y,Z\rangle \unit,
     \Normal(X)\Big\rangle
\]
when \(X\in\partial D\), and of
\[
\Big\langle Y-X - \kappa(Y-p^\pm)
 -\tfrac{\kappa}{2S}\langle X-Y,Z\rangle \unit,
 \Normal(Y)\Big\rangle
\]
when \(Y\in\partial D\). Bearing in mind that \(\|X-Y\|\leq\diam(D)\)
and \(\kappa\leq\delta\), if
\begin{equation}
  \label{eq:drift-normal-bound}
\delta \quad<\quad%\left(\tfrac{1}{\varepsilon}+\tfrac{1}{2S}\right)^{-1}
\frac{\xi}{\diam(D)+R}
\times
\frac{1}{1+\tfrac{1}{2S}\diam(D)}
\end{equation}
then this follows directly when \(\langle X-Y,\Normal(X)\rangle>
\xi\) (in case \(X\in\partial D\)) or when \(\langle
Y-X,\Normal(Y)\rangle> \xi\) (in case \(Y\in\partial D\)).

It remains to consider the case when this does not happen. By choosing \(\xi\) small enough,
we may then ensure that the two Brownian motions are both close to the same segment portion of the boundary.
For the function
\[
\frac{\langle y-x,\nu\rangle}{\|x-y\|}
\]
is continuous on the closed and bounded (and therefore compact)
Euclidean subset
\[
H\quad=\quad
\left\{
(x,y,\nu)\;:\;
x\in D, y\in\partial D, \|x-y\|\geq\varepsilon, \nu \text{ normal at }y
\right\}\,,
\]
and vanishes only on \(\bigcup_i\{(x,y,\nu)\in H: x,y\in\ell_i\}\).
Consequently for any \(\eta>0\) we can find \(\xi>0\) such that 
\[
\langle y-x,\Normal(y)\rangle\quad<\quad \|x-y\|\xi
\]
forces \(x\), \(y\in\ell_i\oplus\ball(0,\eta)\) for some \(i\). If
\(\eta\) is chosen as above then this in turn forces \(p^\pm_i\) to be
the pole for \(x\), \(y\).

We now argue for the case when \(y\in\partial(D)\); the case when \(x\in\partial(D)\) is similar.  We can choose \(\eta>0\) as small as we wish: we therefore require that
\(\sqrt{1-\eta^2/\varepsilon^2}-\eta/\varepsilon>0\).

\begin{figure}[htb]
\includegraphics[width=5in]{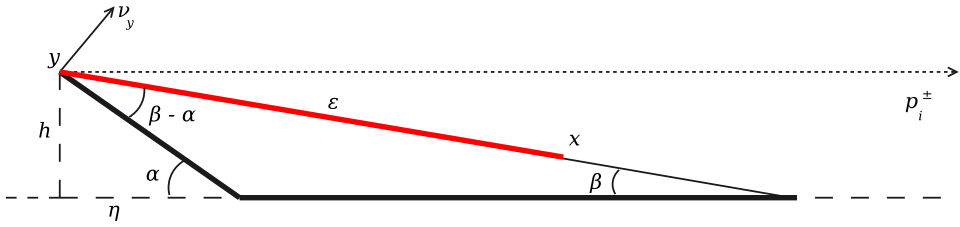}
  \centering
  \caption{Geometric construction underlying analysis of singular
    drift when \(y\in\partial D\) and \(\langle
    y-x,\Normal(y)\rangle<\|x-y\|\xi\).}
\label{fig:figureD}
\end{figure}

Suppose then that \(y\) is at perpendicular distance \(h<\eta\) from
the line through \(\ell_i\) (Figure \ref{fig:figureD}). If \(x\)
is further away from \(\ell_i\) than \(y\), then the singular drift
will certainly be non-negative once \eqref{eq:drift-normal-bound} is
satisfied, since the segment \(y-x\) then makes a smaller angle with
\(\Normal(y)\) than does either \(y-p^\pm_i\) or
\(\tfrac{x+y}{2}-p_i^\pm\). Otherwise (using the angle notation indicated in Figure \ref{fig:figureD}) we require non-negativity of
\[
\varepsilon \sin(\alpha-\beta) - \delta\times(\diam(D)+R)(1
+\frac{1}{2S}\diam(D))\sin\alpha,.
\]
Hence we require
\begin{multline}
  \label{eq:final}
\delta \quad<\quad
\frac{\varepsilon}{(\diam(D)+R)(1+\tfrac{1}{2S}\diam(D)}
\frac{\sin(\alpha-\beta)}{\sin\alpha} \\
\quad=\quad\frac{\varepsilon}{(\diam(D)+R)(1+\tfrac{1}{2S}\diam(D))}
(\cos\beta - \sin\beta\cot\alpha)
\\
\quad\leq\quad
\frac{\varepsilon}{(\diam(D)+R)(1+\tfrac{1}{2S}\diam(D))}
\left(
\sqrt{1-\tfrac{h^2}{\varepsilon^2}}
-
\tfrac{h}{\varepsilon}\tfrac{\eta}{h}
\right)\\
\quad\leq\quad
\frac{\varepsilon}{(\diam(D)+R)(1+\tfrac{1}{2S}\diam(D))}
\left(
\sqrt{1-\tfrac{\eta^2}{\varepsilon^2}}
-
\tfrac{\eta}{\varepsilon}
\right)\,.
\end{multline}
Thus for all sufficiently small \(\delta\) and sufficiently large \(R\) the singular drift is non-positive (the case of \(x\in\partial
D\) following by the same arguments) and so the drift is bounded above
as required, thus completing the proof.

\end{proof}

 \section{Conclusion}\label{sec:conclusion}

The above shows how rather direct potential-theoretic methods permit the deduction of non-existence of shy couplings for reflected Brownian motion in bounded convex domains, so long as either (a) the domain boundary contains no line segments (Theorem \ref{thm:no-segments}) or (b) the domain is planar (Theorem \ref{thm:planar-case}). One is immediately led to the conjecture that there are no shy couplings for reflected Brownian motion in \emph{any} bounded convex domains. In the planar case one can extend the substantial linear portions of the boundary to produce a finite set of poles \(p^\pm\), leading in turn to the key function \(\Phi\) expressed as a minimum of a finite number of simpler functions. Similar constructions will in the general case lead to a continuum of possible poles, contained in an essential continuum of different hyperplanes, and this variety of poles is an obstruction to generalization of the crucial localization analysis in the proof of Theorem \ref{thm:planar-case}. Careful geometric arguments allow progress to be made in the case of bounded convex polytopes, for which attention can be confined to a finite number of hyperplanes and hence to sets of poles forming hypercircles \emph{via} intersection of a large hypersphere with the hyperplanes. However even in this limited special case tedious arguments are required in order to overcome problems arising from intersections of the hypercircles, and we omit the details.

Further variants on the general theme of shy-ness are possible. For example, it is not hard to use the comparison techniques described in \citet[Chapter 2]{JostKendallMoscoRocknerSturm-1997} to show the following: if \(\mathbb{M}\) is a Riemannian manifold with sectional curvatures all bounded above by a finite constant \(\kappa^2\), then there exist \(\varepsilon\)-shy couplings of Brownian motion on \(\mathbb{M}\) for all \(\varepsilon\) satisfying
\[
\varepsilon \quad<\quad 
\min\left\{\tfrac{\pi}{2\kappa}, \text{ injectivity radius of }\mathbb{M}\right\}\,.
\]
These couplings are geometric versions of the perverse coupling described in Section \ref{sec:introduction}.

However the major challenge remains the conjecture of \citet[Open problem 4.5(ii)]{BenjaminiBurdzyChen-2007}, who ask whether there is any simply connected planar domain which supports a shy coupling of reflected Brownian motions. The present work brings us close to understanding shy coupling for convex domains; progress in resolving the \citet{BenjaminiBurdzyChen-2007} conjecture requires development of completely new techniques not dependent on convexity at all.

\section*{Acknowledgements}
I am grateful to Chris Burdzy, Jon Warren, and Saul Jacka for very helpful conversations.

   \bibliographystyle{plainnat}
   \bibliography{abbrev,Shy}

\end{document}